\documentclass[11pt,twoside,a4paper]{article}
\usepackage{graphicx}
\usepackage[english]{babel}
\usepackage{amsmath}
\usepackage{amsthm}
\usepackage{amscd}
\usepackage{pb-diagram}
\usepackage{comment}
\usepackage{amssymb}
\usepackage{subcaption}
\usepackage{color}
\usepackage{enumitem}
\usepackage{faktor}
\usepackage[normalem]{ulem} 

\usepackage{authblk}

\usepackage{algorithm}
\usepackage{algorithmicx}
\usepackage{algpseudocode}

\newcommand{\mo}{\operatorname{mo}}

\newcommand{\inter}{\operatorname{int}}

\newcommand{\img}{\operatorname{im}}

\newcommand{\esol}{\operatorname{eSol}}

\def\articletheorems{
\newtheorem{thm}{Theorem}[section]
\newtheorem{lem}[thm]{Lemma}
\newtheorem{conj}[thm]{Conjecture}
\newtheorem{defn}[thm]{Definition}
\newtheorem{cor}[thm]{Corollary}
\newtheorem{prop}[thm]{Proposition} 
\newtheorem{propdef}[thm]{Proposition and Definition}
\newtheorem{ex}[thm]{Example}
\newtheorem{algo}{Algorithm}[section] % For algorithms in tables
\newtheorem{alg}[thm]{Algorithm}  % For algorithms (MM)
\newtheorem{rem}[thm]{Remark}
}
\articletheorems

\usepackage{amssymb}

\def\refeq#1{\if\workingver y(\ref{#1})-[[#1]]\else(\ref{#1})\fi}
\def\refth#1{\if\workingver y\ref{#1}-[[#1]]\else\ref{#1}\fi}
\def\mylabel#1{\if\workingver y\label{#1}{\bf\ \ [[#1]]\ \ }\else\label{#1}\fi}
\def\mybibitem#1{\if\workingver y\bibitem{#1}{\bf\ \ [[#1]]\ \
}\else\bibitem{#1}\fi}

\def\articletheorems{
\newtheorem{thm}{Theorem}[section]
\newtheorem{lem}[thm]{Lemma}
\newtheorem{conj}[thm]{Conjecture}

\newtheorem{prop}[thm]{Proposition} 

\newtheorem{ex}[thm]{Example}
}

\renewcommand{\emptyset}{\varnothing}
\renewcommand{\rho}{\varrho}
\renewcommand{\epsilon}{\varepsilon}

\def\cA{\text{$\mathcal A$}}

\def\cT{\text{$\mathcal T$}}

\def\cV{\text{$\mathcal V$}}

\newcommand{\cl}{\operatorname{cl}}

\newcommand{\bd}{\operatorname{bd}}

\newcommand{\dom}{\operatorname{dom}}

\renewcommand{\emptyset}{\varnothing}

\newcommand{\Inv}{\operatorname{Inv}}

\def\mathobj#1{\mbox{$#1$}}

\def\ZZ{\mathobj{\mathbb{Z}}}

\renewenvironment{proof}{{\bf Proof:\ }}{\qedsymbol}% To modify or to add to the ctd_defs

\begin{document} 

\title{From Data to Combinatorial Multivector field Through an Optimization-Based Framework}
\author[1]{Dominic Desjardins Côté}
\author[2]{Donald Woukeng}
\affil[1]{Département de Mathématiques, Université de Sherbrooke, Sherbrooke, QC, Canada, dominic.desjardins.cote@usherbrooke.ca}
\affil[2]{Division of Computational Mathematics, Faculty of Mathematics and Computer
Science, Jagiellonian University, ul. St. Lojasiewicza 6, Krakow, 30-348, Poland, donald.woukeng@aims.ac.rw}
\date{}
\maketitle
%%================================%%
%% Sample for structured abstract %%
%%================================%%

%%%%%%%%%%%%%%%%%%%%%%%%%%%%%%%%%%%%%%%%%%%%%%%%%%%%%%%%%%%%%%%%%
%%%%%%%%%%%%%%%%%%%%%%%%%%%%%%%%%%%%%%%%%%%%%%%%%%%%%%%%%%%%%%%%%

\begin{abstract}
    This paper extends and generalizes previous works on constructing combinatorial multivector fields from continuous systems (see \cite{Woukeng_2024}) and the construction of combinatorial vector fields from data (see \cite{arDDC_FiniteVecField}) by introducing an optimization based framework for the construction of combinatorial multivector fields from finite vector field data.  We address key challenges in convexity, computational complexity and resolution, providing theoretical guarantees and practical methodologies for generating combinatorial representation of the dynamics of our data.
\end{abstract}

\section{Introduction}
\subsection{Background and motivation}
Differential equations are essential mathematical tools for modelling dynamical behaviors in biology, physic, science and engineering in general. While they may provide a continuous framework, most of the time their solutions are very hard to obtain analytically, nearly impossible and they require numerical method to approximate solutions with finite constraints. Capturing a global understanding of the system remains very hard.

Combinatorial dynamical systems, introduced as a finite counterpart to continuous systems have emerged as a promising approach for the understanding of the global behavior of a continuous system in a selected region. Following the works of  \cite{forman1998combinatorial} and \cite{lipinski2019conley}, the theory of combinatorial multivector fields was born as a tool to study global dynamics of a continuous system by creating from it a combinatorial mutivector field, that induces itself a combinatorial dynamical system. Combinatorial multivector fields not only preserve most of the key dynamical properties but also permit automated analysis of complex behaviors \cite{Woukeng_2024} with the use of graph theory and through topological invariants such as the Conley index.

 That being said, most of the time, when we want to study some complex systems, we collect data. It happens rarely that these data can be fitted into a differential equation, sometimes we try to fit them into some high order polynomials but then, there is an error generated from that, before studying the differential equation. With our approach, we want be able to say something about the dynamics just from the data. We present two models for multivector field construction using an optimization-based framework and explore their implications for dynamical systems analysis from the data.
 
\subsection{Overview of main results}\label{ss:OverviewMain}

            For the overview, we consider only the optimization problem (\ref{eq:MinOptMv}). To setup the optimization problem, we suppose that we have a simplicial complex $ K $, and a map $ V : K \to \mathbb{R}^n$. Let $ z(\sigma, \tau) \in \mathbb{Z}_2 $ be our variables such that $ \sigma < \tau \in K $ where $ \tau $ is a toplex. If $ z(\sigma, \tau) =1 $, then $ \sigma $ and $ \tau $ are in the same multivector. We assign a cost $ c_i $ to each variables which depends of $ \sigma $, $ \tau$, and the value of $ V(\sigma) $. If $ \sigma $ and $ \tau $ should be in the same multivector, then the cost of $ z(\sigma, \tau) $ is low. We obtain the following optimization problem :

            \begin{equation}    \label{eq:Intro}
                \begin{aligned}
            		& \underset{z(\sigma, \tau) \in \{ 0, 1 \} }{\text{minimize}}
            		& & f(\vec{z}) = \vec{c} \cdot \vec{z} \\
            		& \text{subject to}
                    & & \sum_{\sigma < \tau_i} z(\sigma, \tau_i) = 1 \qquad \text{for } \sigma \in K \setminus T \\
                    & & & z(\sigma_i, \tau) - z(\sigma_j, \tau) \leq 0 \, \text{for } \sigma_i <\sigma_j < \tau.
        		\end{aligned}
            \end{equation}
            The first set of constraints ensure that each simplex is in a single multivector. The second set set of constraints guarantee the convexity property of combinatorial multivector field.  We apply this method to a simple problem. Consider the following dynamical system :
            \begin{equation}\label{eq:dynSysRepOrb}
                \begin{cases}
                    \frac{dx}{dt} = y + x(x^2 + y^2 - 1) \\
                    \frac{dy}{dt} = -x +y(x^2 + y^2 - 1)
                \end{cases}.  
            \end{equation}
            This system has an attractive fixed points and a repulsive orbit. We choose a set of $ 1000 $ data points randomly in $ [-3 ,3] \times [-3, ,3] $. We apply a k-means approach to reduce the number of data points to $ 100 $ clusters. We build a Delaunay complex on the cluster. For the map $ V : K \to \mathbb{R}^2 $, we compute the $0$-simplices with (\ref{eq:dynSysRepOrb}). For a $d$-simplex $\sigma $, $ V(\sigma) $ is the average of $V(\rho_i)$ where $ \rho_i$ is $0$-simplex of $ \sigma $ for all $i$. From the up-left figure of (\ref{fig:ExModel2}), we see the optimal solution of (\ref{eq:MinOptMvMatrix}). If there is an arrow from the barycenter of $ \sigma $ to the barycenter of $ \tau $, then $ z(\sigma, \tau) = 1 $. Otherwise, $ z(\sigma, \tau) = 0 $. From the up-right figure of (\ref{fig:ExModel2}), we obtain the combinatorial multivector field from (\ref{eq:Model2MultVec}). We obtain a single critical multivector $ V $ with the number $47$. This multivector $ V $ represent well the dynamics of an attractive fixed point. The down-left figure of (\ref{fig:ExModel2}), we have the only strongly connected component $S$. We have that the exit set of $ S $ is the border of an annulus, and it represent well the dynamics of a repulsive orbit. The down-right figure of (\ref{fig:ExModel2}), we have the gradient component of the combinatorial multivector field.
            
            %\textbf{Todo:}\textit{ Maybe I am wrong but isn't the fixed point also a part of the reccurent component of our multivector field?}
            % Answer : No, from the fixed multivector, we cannot come back to the recurrent component

%%%%%%%%%%%%%%%%%%%%%%%%%%%%%%%%%%%%%%%%%%%%%%%%
         \begin{figure}
  	         \center
             \includegraphics[height=4cm, width=6cm, scale=1.00, angle=0 ]{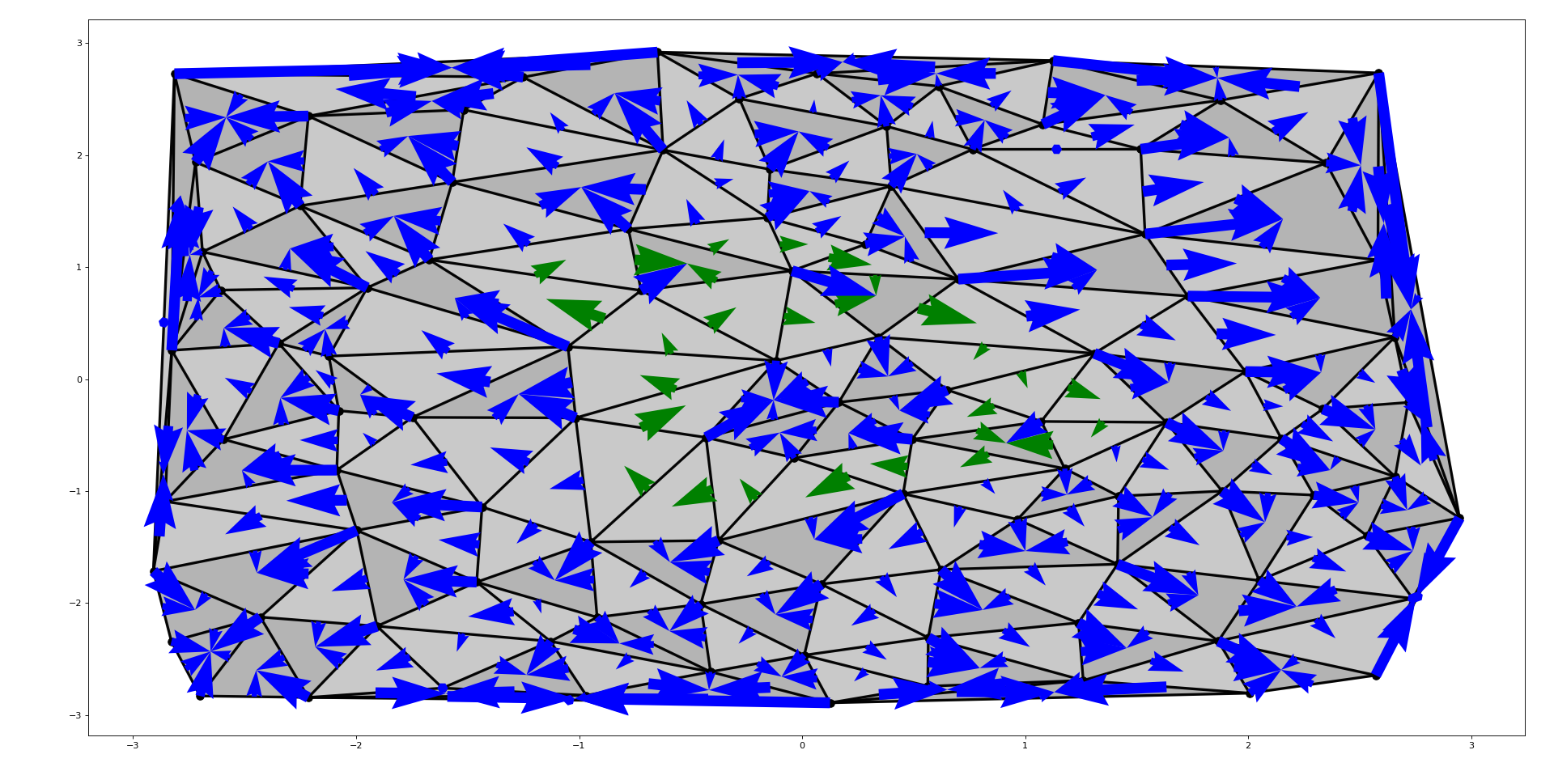}
             \includegraphics[height=4cm, width=6cm, scale=1.00, angle=0 ]{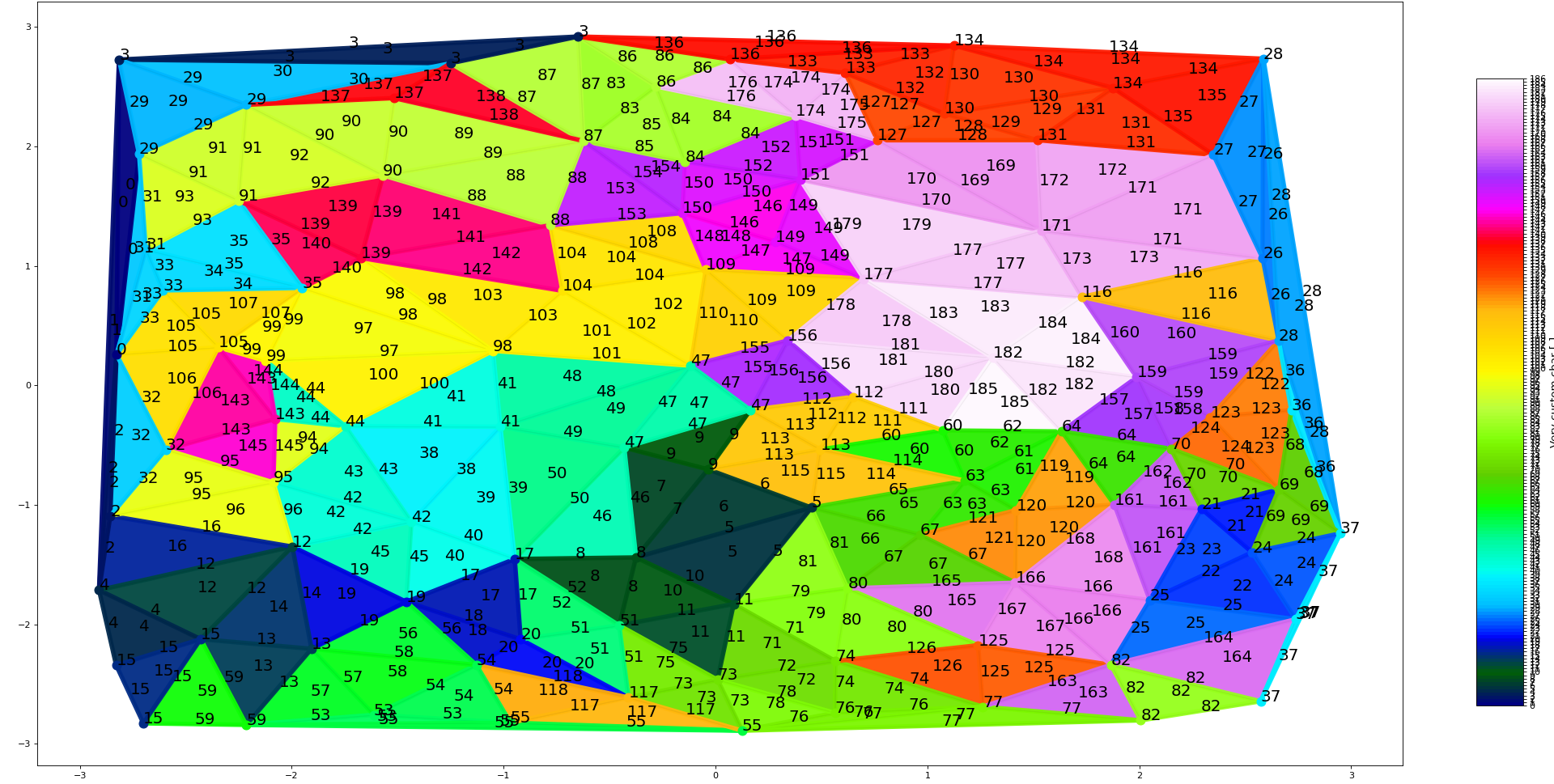}
             \includegraphics[height=4cm, width=6cm, scale=1.00, angle=0 ]{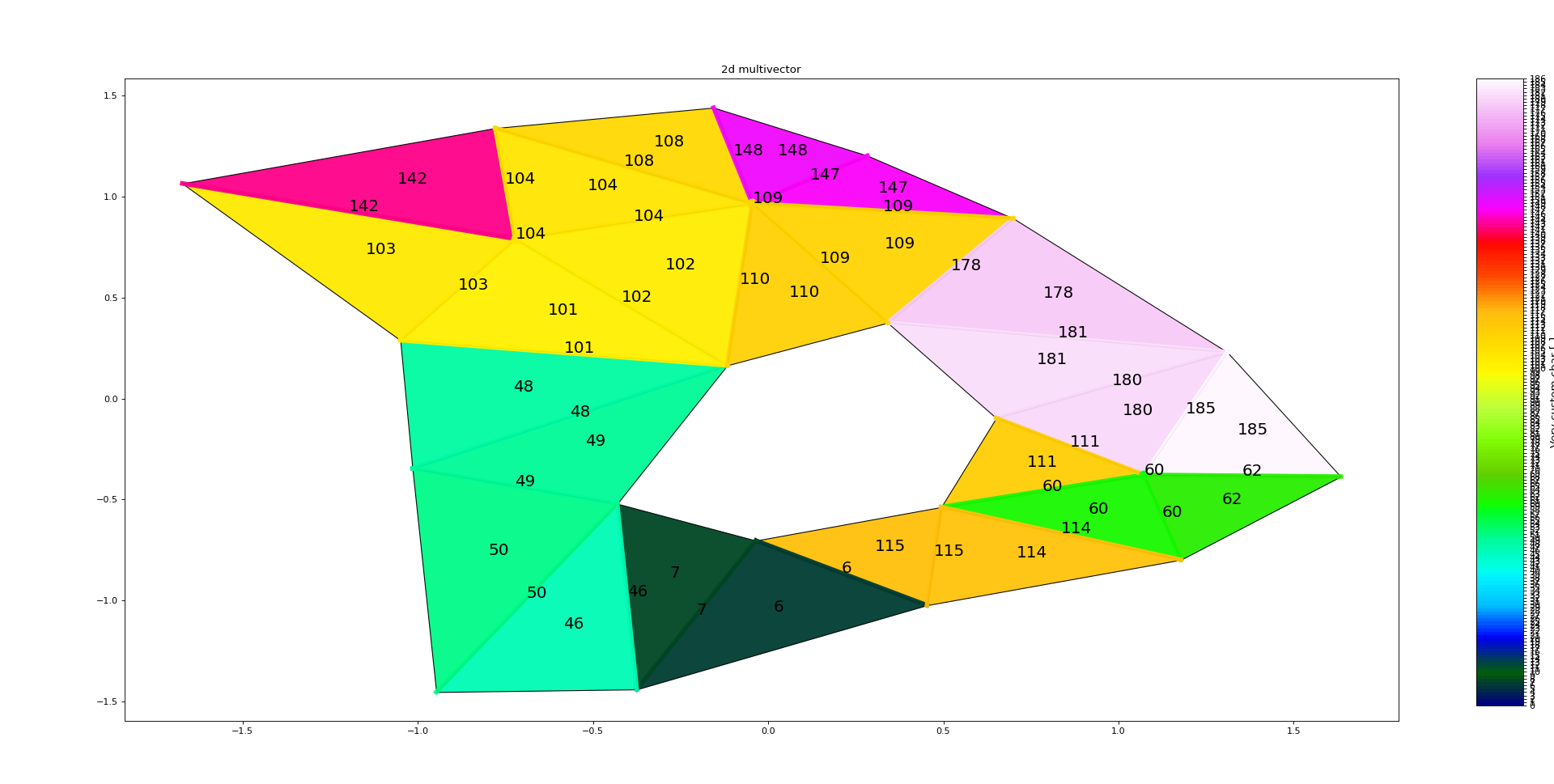}
             \includegraphics[height=4cm, width=6cm, scale=1.00, angle=0 ]{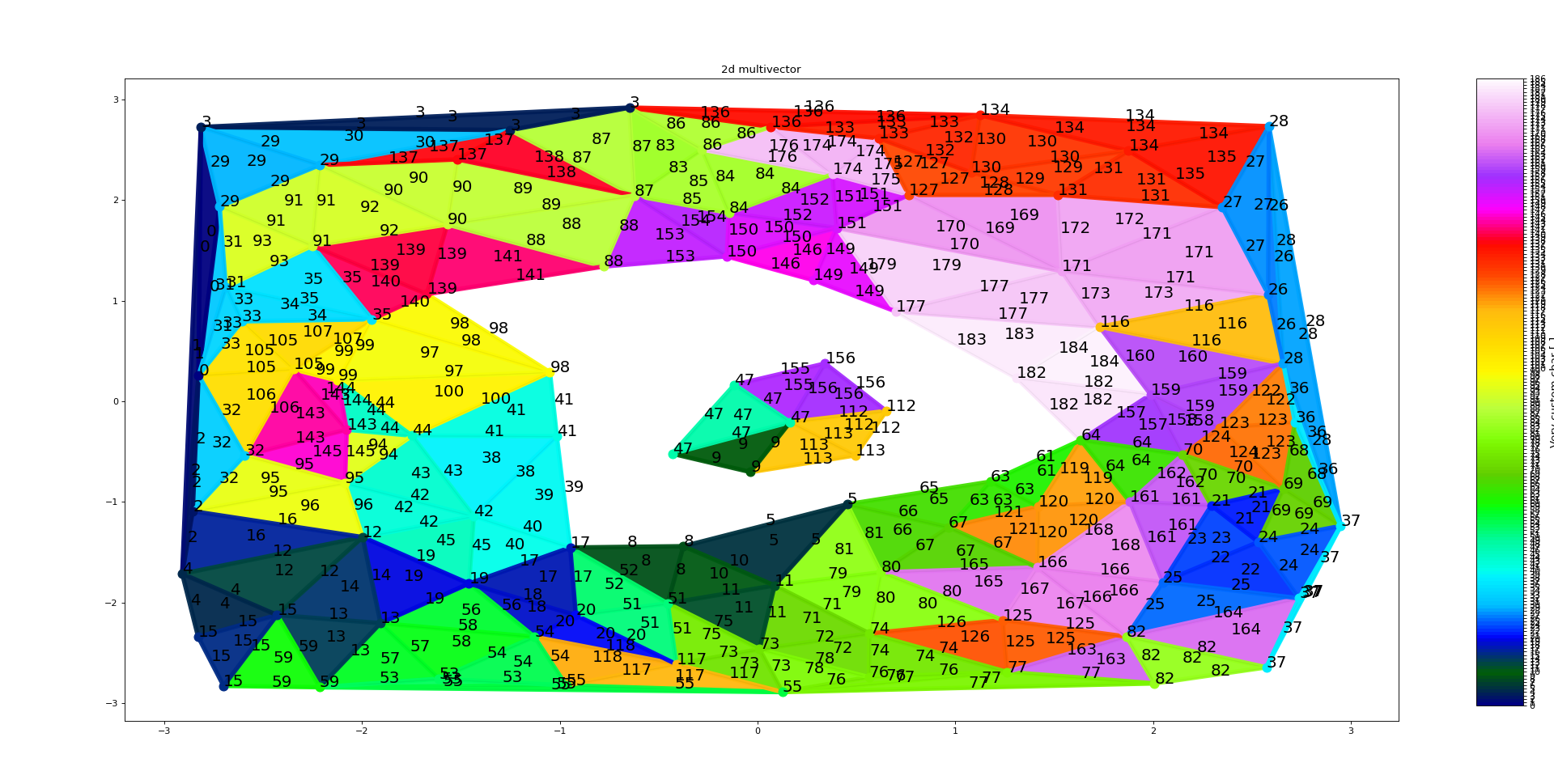}	             
             \caption{ The combinatorial multivector field obtained from the optimization problem (\ref{eq:Intro}) with data from the dynamical system (\ref{eq:dynSysRepOrb}).}	
             \label{fig:ExModel2}
     \end{figure}

\section{Preliminaries}

\subsection{Relations and posets}
Let $X$ be a set. 
A \emph{binary relation} on {$ X$} is a subset $R \subset X \times X$. 
We use the classical notation $xRy$ to denote $(x,y) \in R$.

We recall that a reflexive, antisymmetric and transitive relation $\leq$ is a \emph{partial order} and the pair $(X,\leq)$ is a \emph{poset}.
We recall that a set $A \subset X$ for a poset $(X,\leq)$ is an \emph{upper set} if $\{ z \in X  \mid \exists_{x\in A}\, x \leq z \} \subset A$.
Analogously, $A\subset X$ is a \emph{lower set} if $\{ z \in X  \mid \exists_{x\in A}\, z \leq x \} \subset A$.
A set $A\subset X$ is \emph{convex} with respect to poset $(X,\leq)$ if 
    for every $x,z\in A$ and $y\in X$ such that $x\leq y\leq z$ it follows that $y\in A$.
Equivalently, a convex set is an intersection of a lower set and an upper set.

We say that a \emph{relation} $R$ is an \emph{equivalence relation} if it is reflexive, symmetric, and transitive.
We denote the \emph{equivalence class} of $x$ in $R$ by $[x]_{R}:=\{y\in X\mid x R y\}$.

A \emph{partition} of a space $X$ is a family $\cV$ of non-empty subsets of $X$ such that for every $A,B\in\cV$ we have $ A \cap B = \emptyset$ and $\bigcup \cV = X$.
In particular, partition $\cV$ induces an equivalence relation $R$ defined by $xRy$ if there exists $A\in \cV$ such that $x,y\in A$. 
We denote the equivalence class of a point $x\in X$ induced by partition $\cV$ by $[x]_\cV$.

\subsection{Topological spaces}
Given a topology $\mathcal{T}$ on $X$, we call $(X, \mathcal{T} )$ a topological space. 
When the topology $\mathcal{T}$ is clear from the context, we also refer to $X$ as a topological space. 
We denote the interior, closure, and boundary of a $A\subset X$ with respect to $\cT$ by $\inter_\cT A$, $\cl_\cT A$, and $\bd_\cT A$, respectively.
The \emph{mouth} of A is defined as $\mathrm{mo}_\mathcal{T} A := \mathrm{cl}_\mathcal{T} A \setminus A$.
If the topology is known from the context, we write, e.g., $\cl_X A$, or we skip the subscript completely.
Let $Y\subset X$. The induced topology on $Y$ is defined as $\cT_Y:=\{U\cap Y\mid U\in\cT\}$.
The closure of $A\subset Y$ with respect to the induced topology on $Y$ will be denoted by $\cl_{\cT_Y} A$ or simply by $\cl_Y A$.
Similarly, for the interior, boundary, and mouth.

A set $A\subset X$ is $\emph{locally~closed}$ if every $x \in A$
admits a neighborhood $U$ in $X$ such that $A \cap U$ is closed in $U$.

\begin{prop}
\cite[Problem 2.7.1]{En1989}. 

    Assume $A$ is a subset of a topological space $X$.
Then the following conditions are equivalent.
\begin{enumerate}[label=(\roman*)]
    \item $A$ is locally closed,
    \item $\mo A=\cl A\setminus A$ is closed,
    \item $A$ is a difference of two closed subsets of $X$,
    \item  $A$ is an intersection of an open set in $X$ and a closed set in $X$.
\end{enumerate}

It is easy to see that the intersection of a finite family of locally closed subsets is locally closed.
\end{prop}

We recall that a topological space is a $T_0$ topological space if, for every pair of distinct points of $X$, at least one of them has a neighborhood not containing the other.
We will be particularly interested in $T_0$ finite topological spaces.
By the Alexandroff theorem \cite{Alexandroff_ftop} we can identify them with partial orders.
\begin{thm}\label{thm:finite_top} \cite{Alexandroff_ftop}
For a finite poset $(P,\leq)$ the family $\cT_{\leq}$ of upper sets of $\leq$ is a $T_0$ topology on $P$. For a finite $T_0$ topological space $(X,\cT)$, $x,y\in X$ the relation $x \leq_{\cT} y$ defined by $x \in \cl_{\cT} \{ y \}$ is a partial order on $X$. Moreover, the two associations that relate $T_0$ topologies and partial orders are mutually inverse.
\end{thm}

It follows that in this setting all topological concepts can be expressed in terms of partial order concepts and vice versa.
In particular, given a finite $T_0$ topological space $(X, \cT)$ with the associated poset $(X,\leq)$, we get for $A\subset X$
\begin{align}
    \cl_\cT A &:= \{x\in X\mid \exists_{a\in A}\ x\leq a\}, \label{eq:closure_poset}\\
    \bd_\cT A &:= \{x\in X \mid \exists_{a\in A,\ b\in X\setminus A}\ x\leq a \text{ and } x\leq b\}\label{eq:boundary_poset}.
\end{align}

We emphasize that in the setting of finite topological spaces local closedness and convexity coincide \cite[Proposition 1.4.10]{lipinski_phd}.
Hence, $A$ is a locally closed set in a $T_0$ topology iff it is convex with respect to the associated poset. 
We will use these two concepts exchangeably throughout the paper; locally closed when we try to emphasize the topological context and convex when we focus on the algorithmical or the combinatorial aspect.

\subsection{Combinatorial multivector fields}

 All of the definitions in this subsection can be found in \cite{lipinski2019conley}.

Let X be a finite topological space. A \emph{combinatorial multivector} or briefly a \emph{multivector} is a
locally closed subset $V \subset X$. 
A \emph{combinatorial multivector field} on $X$, or briefly a \emph{multivector field}, is a partition $\mathcal{V}$ of $X$ into multivectors.

Since $\cV$ is a partition, we can denote by $[x]_{\mathcal{V}}$ the unique multivector in $\cV$ that contains $x\in X$. 
If the multivector field $\cV$ is clear from the context, we write briefly $[x]$.
We say that a multivector $V\in\cV$ is \emph{critical} if the relative singular homology $H(\cl V, \mo V )$ is non-trivial.
A multivector $V$ which is not critical is called \emph{regular}. 
We say that a set $A \subset X$ is \emph{$\cV$-compatible} if for every $x \in X$ either $[x] \cap A = \emptyset$ or $[x] \subset A$.

Multivector field $\mathcal{V}$ on $X$ induces a  multivalued map $\Pi_{\mathcal{V}} : X \multimap X$ given by 
\begin{align}\label{eq:piv}
    \Pi_{\mathcal{V}}(x)=[x]_{\mathcal{V}}\cup \cl x~.
\end{align}

We consider a combinatorial dynamical system given by the iterates of~$\Pi_\cV$.

A \emph{solution} of a combinatorial dynamical system $\Pi_\cV:X\multimap X$ in $A\subset X$ is a partial map $\varphi:\mathbb{Z}\nrightarrow A$ whose domain, denoted $\dom \varphi$,  is a $\mathbb{Z}$-interval and for any $i,i+1\in \dom \varphi$ the inclusion $\varphi(i+1)\in \Pi_\cV(\varphi(i))$ holds. Let us denote by $\textup{Sol}(A)$ the set of all solutions $\varphi$ such that $\img\varphi\subset A$. $\textup{Sol}(X)$ is the set of all solution of $\Pi_\cV$.
If $\dom\varphi$ is a bounded interval then we say that $\varphi$ is a \emph{path}.
If $\dom\varphi=\ZZ$ then $\varphi$ is a \emph{full solution}. 

A full solution $\varphi : \mathbb{Z} \rightarrow X$ is \emph{left-essential} (respectively \emph{right-essential})
if for every regular $x \in \img\varphi$ the set $\{ t \in \mathbb{Z}\mid \varphi(t) \notin [x]_{\mathcal{V}} \}$ is left-infinite (respectively right-infinite). 
We say that $\varphi$ is \emph{essential} if it is both left- and right-essential.
The collection of all essential solutions $\varphi$ such that $\img\varphi\subset A$ is denoted by $\textup{eSol}(A)$.
%A solution $\varphi  \mathbb{Z} \rightarrow X $ is maximal if either $\varphi$ is an essential solution or a finite path that cannot be extended to an essential solution.

The \emph{invariant} part of a set $A\subset X$ is 
$\Inv A := \bigcup \{\img \varphi\mid \varphi\in\esol(A)\}$.
In particular, if $\textup{Inv} A = A$ we say that $A$ is an \emph{invariant set} for a multivector field $\mathcal{V}$.

A closed set $N\subset X$ \emph{isolates} invariant set $S \subset N$ if the following conditions are satisfied:
\begin{enumerate}[label=(\roman*)]
    \item every path in $N$ with endpoints in $S$ is a path in $S$,
    \item $\Pi_{\mathcal{V}}(S) \subset N$.
\end{enumerate}
In this case, $N$ is an \emph{isolating set} for $S$. 
If an invariant set $S$ admits an isolating set then we say that $S$ is an \emph{isolated invariant set}.
The \emph{homological Conley index} of an isolated invariant set $S$ is defined as $\operatorname{Con}(S):=H(\cl S,\mo S)$.

Let $A \subset X$.
By $\bigl\langle A \bigl\rangle_{\mathcal{V}}$ we denote the intersection of 
all locally closed and $\cV$-compatible sets in $X$ containing $A$.
We call this set the $\mathcal{V}$-\emph{hull} of $A$. 
The combinatorial $\alpha$-\emph{limit set} and $\omega$-\emph{limit set} for a full solution $\varphi$ are defined as
\begin{align*}
    & \alpha(\varphi) := \Bigl\langle \bigcap\limits_{t \in \mathbb{Z}^-}\varphi((-\infty,t]) \Bigl\rangle_\mathcal{V}\ , \\
    & \omega (\varphi) := \Bigl\langle \bigcap\limits_{t \in \mathbb{Z}^+}\varphi([t, \infty))  \Bigl\rangle_\mathcal{V}\ .
\end{align*}

Let $S\subset X$ be a $\cV$-compatible, invariant set.
Then, a finite collection $\mathcal{M}=\{M_p\subset S\mid p\in\mathbb{P}\}$ is called a \emph{Morse decomposition} of $S$ if there exists a finite poset $(\mathbb{P},\le)$ such that the following conditions are satisfied:
\begin{enumerate}[label=(\roman*)]
    \item $\mathcal{M}$ is a family of mutually disjoint, isolated invariant subsets of $S$,
    \item for every $\varphi\in\esol(S)$ either $\img\varphi \subset M_r$  for an $r \in \mathbb{P}$
or there exist $p, q \in \mathbb{P}$ such that $q > p$,    $\alpha(\varphi)\subset M_q \text{, and } \omega(\varphi)\subset M_p$.
 \end{enumerate}
We refer to the elements of $\mathcal{M}$ as \emph{Morse sets}.

%\subsection{Optimization}

\section{Construction of a combinatorial multivector field from data}

    We show a construction of a combinatorial multivector field by solving a linear optimization problem with binary variables. Let $ K $ be a simplicial complex, $ T $ the set of toplexes in $K$, and $N$ be the number of simplices. We also have a map $ V : K \to \mathbb{R}^d $ that assign a vector to each simplex.

    \subsection{Optimization problem : Objective function}
        First, we define the variables of our optimization problem. Let $ z(\sigma, \tau)  \in \{ 0, 1 \} $ be a variable where $ \sigma, \tau \in K $. If $ z(\sigma, \tau) = 1 $, then $ \sigma $ and $ \tau $ are in the same multivector. In figures, we draw an arrow from $ \sigma $ to $ \tau $, if $ z(\sigma, \tau) = 1 $, and a simplex $ \sigma $ is red, if $ z(\sigma, \sigma) = 1 $. The idea is to assign a cost $c$ to each variable $ z $. We want to define a minimization problem. Therefore, if $ c $ is low, then $ \sigma $ and $ \tau $ should be in the same multivector. Let $ c_j $ be a real value associated to the variable $ z_j(\sigma, \tau) $. We also need the following map $ W(\sigma, \tau) : K \times K \to \mathbb{R}^n $ given by $ W(\sigma, \tau) = b(\tau) - b(\sigma) $ where $ b(\sigma) $ is the barycenter of $ \sigma $. We are going to compare two vectors with the cosine similarity 
        \begin{equation*}
            d(u, v) = 1 - \frac{u\cdot v}{\| u \| \| v \|}.
        \end{equation*}
        We have three different costs $ c_j$ for a variable $ z_j(\sigma, \tau )$ to consider. If $ V(\sigma) \neq \vec{0} $ and $ \sigma \neq \tau $, then we set $ c_j = d(V(\sigma), W(\sigma, \tau)) $. If $ V(\sigma) = \vec{0} $, then $ d(V(\sigma), u) = 1$ for any vector $u$. This can generate some errors in the final result, and we obtain better result to set $ c_j = 2 $. If $ \sigma = \tau $, then $ c_j $ is associated to $ z(\sigma, \sigma) $ represents the cost to have a simplex alone in a multivector. We assign a parameter $ \alpha \in \mathbb{R} $.  In summary, the cost $ c_j $ is define by
        \begin{equation}\label{eq:CostVect}
            c_j := \begin{cases}
                    d(V(\sigma), W(\sigma, \tau)) & \text{If } V(\sigma) \neq \Vec{0} \text{ and } \sigma \neq \tau \\
                    2 & \text{If } V(\sigma) = \Vec{0} \\
                    \alpha & \text{Otherwise }
            \end{cases}.
        \end{equation}

        We will minimize the function $ f(\vec{z}) = \vec{z} \cdot \vec{c} $.  
        
        \subsection{Model 1 : Generalization of the CDS Model}
        
            We generalize the case of combinatorial dynamical system in the sense of Forman from \cite{arDDC_FiniteVecField}. Let $ D_{n \times m} $ be a binary matrix with $ n = \vert K \vert $ and $ m $ the number of variables. We assign an index $i$ for each simplex. Let $ \sigma_{i} $ be associated with the $i$th row and $ z_j(\beta, \tau) $ associated to the $j$th column. The entries of $ D $ are :
             \begin{equation*}
                D_{i, j} := \begin{cases}
                    1 & \text{if } \sigma_i = \beta \text{ or } \sigma_i = \tau, \\
                    0 & \text{otherwise}
            \end{cases}.
            \end{equation*}

            We want to have that all simplices are at least in a matching. We add the constraint $ D\vec{z} \geq \vec{1} $. We need to add more constraints to satisfy the convexity of multivector. Therefore, we add three new type of constraints: 
            \begin{gather}
                z(\sigma_i, \tau) - z(\sigma_j, \tau) \leq 0,\label{eq:ConvConst1} \\
                z(\sigma_i, \tau) + z(\sigma_j, \tau) - z(\sigma_i, \sigma_j) \leq 1, \label{eq:ConvConst2}  \\
                z(\sigma_i, \sigma_j) + z(\sigma_j, \tau) - z(\sigma_i, \tau) \label{eq:ConvConst3}  \leq 1
            \end{gather}.
            for each triplet $ \sigma_i < \sigma_j < \tau $.

            \begin{ex}  \label{ex:constConvex}
                Let $ x $ be a 0-simplex, $ \sigma_1, \sigma $ be $1$-simplices, and $ \tau $ be a $2$-simplex such that $ x < \sigma_1, \sigma_2 < \tau$ which are represented at the Figure (\ref{fig:exModel1ConvexCstr}). If $ z(x, \tau) = 1 $, then $ z(\sigma_1, \tau) = 1 = z(\sigma_2, \tau)$ by (\ref{eq:ConvConst1}). We have $ z(x, \tau) + z(\sigma_1, \tau) = 2 $. To satisfy the constraint (\ref{eq:ConvConst2}), we need to have $ z(\sigma_1, \tau)  = 1 $.

                Let $ \rho \neq \tau $ be a $2$-simplex such that $ x < \sigma_1 < \rho$. We suppose that $ z(\sigma_1, \rho) = 1$ and $ z(x, \sigma_1) + z(\sigma_1, \rho) = 2 $. To satisfy (\ref{eq:ConvConst3}), we need to have $ z(x, \rho) = 1 $. 
            \end{ex}
            %%%%%%%%%%%%%%%%%%%%%%%%%%%%%%%%%%%%%
            \begin{figure}
      	         \center
                 \includegraphics[height=6cm, width=7.5cm, scale=1.00, angle=0 ]{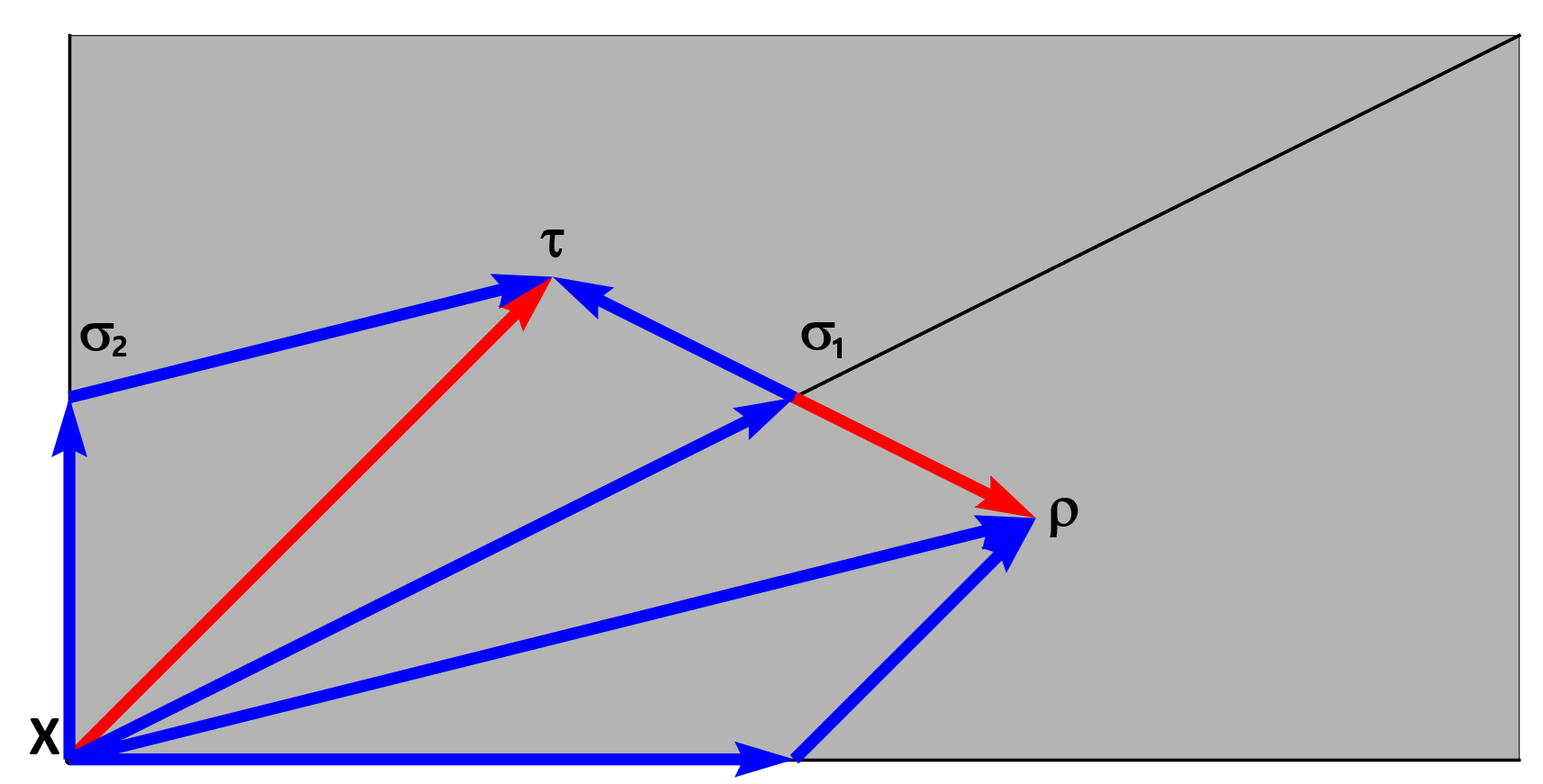}         
                 \caption{The red arrows are the the suppositions of the Example (\ref{ex:constConvex}).  By the constraints (\ref{eq:ConvConst1}) (\ref{eq:ConvConst2}), and (\ref{eq:ConvConst3}), we need to set some variables to $1$ which are represented by blue arrows. }	
                 \label{fig:exModel1ConvexCstr}
             \end{figure}
            %%%%%%%%%%%%%%%%%%%%%%%%%%%%%%%%%%%

            We obtain the following linear minimization problem with binary variables :
            \begin{equation}    \label{eq:MinOptGenCDS}
                \begin{aligned}
            		& \underset{z(\sigma, \tau) \in \{ 0, 1 \} }{\text{minimize}}
            		& & f(\vec{z}) = \vec{c} \cdot \vec{z} \\
            		& \text{subject to}
                    & & Dz \leq \vec{1}, \\
                    & & & z(\sigma_i, \tau) - z(\sigma_j, \tau) \leq 0 \, & \text{for } \sigma_i < \sigma_j < \tau,\\
                    & & & z(\sigma_i, \tau) + z(\sigma_j, \tau) - z(\sigma_i, \sigma_j) \leq 1 \,&  \text{for } \sigma_i < \sigma_j < \tau,\\
                    & & & z(\sigma_i, \sigma_j) + z(\sigma_j, \tau) - z(\sigma_i, \tau) \leq 1 \, & \text{for } \sigma_i < \sigma_j < \tau.
        		\end{aligned}
            \end{equation}
    
            First, there exist solutions of (\ref{eq:MinOptGenCDS}) that do not induce a combinatorial multivector field. Because, the convexity constraints are only satisfied for triplet $ \sigma_i < \sigma_j < \tau $.
            \begin{ex}
               Consider the simplicial complex $ K $ from the Figure (\ref{fig:ceModel1Convex}) with a $\vec{z}$ defined by the arrows. It is a feasible solution of the optimization problem (\ref{eq:MinOptGenCDS}). The set $ V := \{ [x_1],$ $ [x_1, x_2], [x_1, x_4], [x_1, x_2, x_4], [x_2], [x_2, x_4], [x_2, x_3], [x_2,$ $  x_3, x_4], [x_3, x_4], [x_1, x_3, x_4]  \} $  is not convex. Because we have $[x_1] $ and $ [x_1, x_3,$ $ x_4] $ in the same set, and we need to add the simplex $ [x_1, x_3] $ in $ V $ to satisfy the convexity.
            \end{ex}

            %%%%%%%%%%%%%%%%%%%%%%%%%%%%%%%%%%%%%%%%%%%%%%%
            \begin{figure}
      	         \center
                 \includegraphics[height=6cm, width=8.5cm, scale=1.00, angle=0 ]{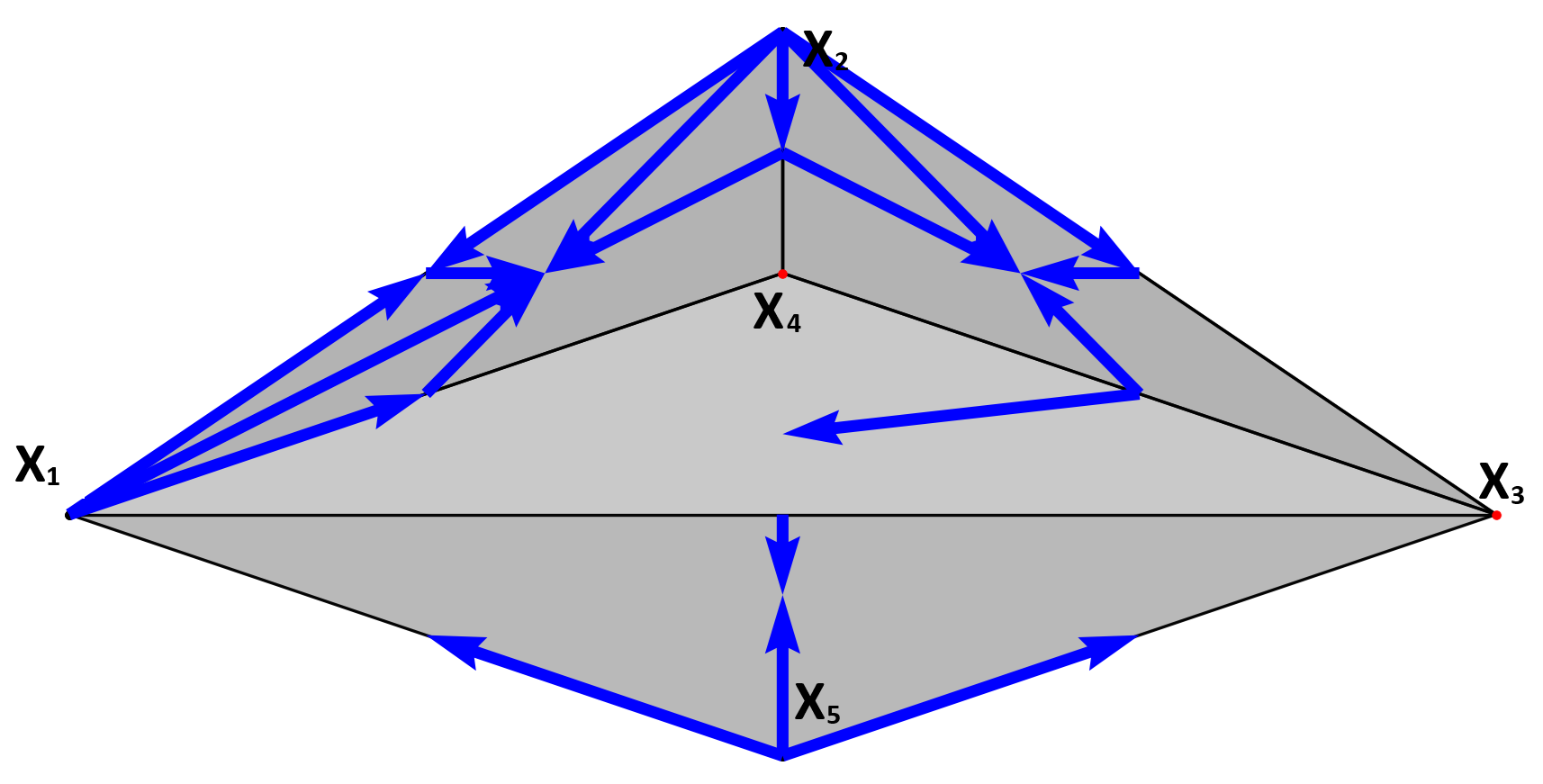}         
                 \caption{This is a feasible solution of (\ref{eq:MinOptGenCDS}) that does not induce a combinatorial vector field. The convexity condition is not satisfied.}	
                 \label{fig:ceModel1Convex}
             \end{figure}

            %%%%%%%%%%%%%%%%%%%%%%%%%%%%%%%%%%%%%%%%%%%%
           
            We need to apply a post-processing method to the solution of (\ref{eq:MinOptGenCDS}) to ensure the convexity property. For a non-convex set, we can merge with another set or divide it in two. 
            
            There is still an advantage for (\ref{eq:MinOptGenCDS}). Any combinatorial multivector field can be obtain from a solution of (\ref{eq:MinOptGenCDS}). 
             \begin{thm}
                 Let $ \mathcal{V} $ a combinatorial multivector field on $ K $. There exists a solution $ \vec{z} $ of (\ref{eq:MinOptGenCDS}) such that it induce the same combinatorial multivector field as $ \mathcal{V} $.
             \end{thm}
             \begin{proof}
                From a combinatorial multivector field $ \mathcal{V} $, we define a $ \vec{z} $ as follows. If $ \vert V \vert = 1 $, then, we set $ z(\sigma, \sigma) = 1 $ for $ \sigma \in V $. If $ \vert V \vert > 1 $, then, for $ \sigma, \tau \in V  $ and $ \sigma < \tau $, we set $ z(\sigma, \tau) = 1 $. We verify that $\vec{z} $ is a feasible solution of (\ref{eq:MinOptGenCDS}).
    
                Since $ \mathcal{V} $ is a partition of $K$, then each simplex is a in multivector. Therefore, $ D\vec{z} \geq \vec{1} $.
    
                For the convexity constraint, the variables $ z(\sigma, \sigma) $ are not in. We need only to consider the variables $ z(\sigma, \tau) $ with $ \sigma \neq \tau$. If $ \vert V \vert \leq 2 $, the constraints are satisfied trivially. 

                If $ \vert V \vert > 2 $. For any triplets $ \sigma_i < \sigma_j < \tau $ in $ V $, we have $ z(\sigma_i, \tau) = 1 = z(\sigma_j, \tau) = z(\sigma_i, \sigma_j) $ by construction. Therefore, we obtain equality for the constraints (\ref{eq:ConvConst1}), (\ref{eq:ConvConst2}) and (\ref{eq:ConvConst3}). We have that $ \vec{z} $ is a solution of (\ref{eq:MinOptGenCDS}). 
             \end{proof}
    
            There is an other problem. By minimizing the cost function, we obtain a low count of variables $ z(\sigma, \tau) = 1 $, because all costs are positives. Therefore, we obtain really small multivector, and it can also satisfy the definition of combinatorial dynamical system in the sense of Forman. A solution to this problem is to set the cost of good matching to be negative. We add a new parameter $ \beta \in \mathbb{R} $. We remove the value of $ \beta $ for some costs. For a variable $ z_j(\sigma, \tau) $, the new cost is :
            \begin{equation*}
                c'_j(\sigma, \tau) := \begin{cases}
                    d(V(\sigma), W(\sigma, \tau)) - \beta & \text{If } \sigma \neq \tau \\
                    \alpha & \text{Otherwise}
                \end{cases}
            \end{equation*}
            Let $ g(z) = c' \cdot z $ be the new objective function. If $ \beta = 0 $, we have the previous objective function $ f(z) $. If $ \beta \geq 2 $, then every cost are negatives for variables $z(\sigma, \tau) $. Therefore, the optimal solution will induce a combinatorial multivector field where every simplices of the same connected component are in the same multivector. 

            \begin{lem}\label{lem:ParamCMVF}
                Let $ K $ be a finite simplicial complex. We consider the optimization problem (\ref{eq:MinOptGenCDS}) where all costs are strictly positives.
                \begin{itemize}
                    \item There exist an $ \alpha \in [0,2] $ such that the optimal solution of (\ref{eq:MinOptGenCDS}) induces a combinatorial multivector field.
                    \item If $ \beta = 0 $, then the optimal solution of (\ref{eq:MinOptGenCDS}) induces a combinatorial multivector field.
                \end{itemize}
            \end{lem}
            \begin{proof}
                Let $ l = \min_{\sigma < \tau} c(\sigma, \tau) $. We choose $ \alpha < \frac{l}{2} $. For any $ \sigma \leq \tau$, we have that for $ c(\sigma, \sigma) + c(\tau, \tau) < l \leq z(\sigma, \tau) $. We consider the solution $ \vec{z}_1 $ where $ z(\sigma, \sigma) = 1 $ for all $ \sigma $, and $ z_1 $ is a feasible solution of (\ref{eq:MinOptGenCDS}). For any other solutions $ \vec{z}_2 $, we have that $ f(\vec{z}_1) < f(\vec{z}_2) $. Finally, $ \vec{z}_1$ induce a combinatorial multivector field $\mathcal{V}$ where $ \vert V \vert = 1 $ for all $ V \in \mathcal{V} $.  

                Let $ \beta = 0 $. Let $ \vec{z} $ a solution of (\ref{eq:MinOptGenCDS}), and $ \mathcal{V} $ the induced partition by $ \vec{z} $. Let $ d_{V} = \max_{\sigma} \dim(\sigma) - \min_{\sigma \in V} $. If $ d_{V} \leq 1 $ for all $ V $, then $ \mathcal{V} $ is a combinatorial multivector field. Because, If $ d_{V} = 0 $, then $ \vert V \vert = 1 $. If $ d_{V} = 1 $, then for all simplex $ \sigma, \tau \in V $ such that $ \sigma < \tau $, then it does not exist an $ \rho $ such that $ \sigma < \rho < \tau $. Therefore, $ V $ is always convex. We want to show that if there exists $  V $ such that $ d_V > 2 $, then the solution $ \vec{z} $ is not a global minimum of (\ref{eq:MinOptGenCDS}). 

                We suppose there exists a $ V $ such that $ d_V > 1 $. There exists $ \sigma < \tau \in V $ such that $ \dim \tau - \dim \sigma = 2 $, and $ z(\sigma, \tau) = 1$. We suppose that $ \tau $ is maximum in $ V $. Since $ K $ is a simplicial complex, there exists $ \rho_1, \rho_2 \in V $ such that $ \sigma < \rho_1, \rho_2 < \tau $. By the constraint of (\ref{eq:MinOptGenCDS}), we have that $ z(\sigma, \rho_i) = 1 = z(\rho_i, \tau) $ for $ i = 1, 2$. The set $V_1 = \cl \rho_i \cap V $ is convex, and the set $ V_2 = V \setminus V_1 $ is also convex. We define a new solution $ \vec{z_1} $ such that $ z(\sigma, \tau) = z_1(\sigma, \tau) $ except for $ z_1(\rho, \tau) = 0 $ for $ \rho \in V_1 $ and $ \tau \in V_2 $. We obtain that $ f(\vec{z_1}) \leq f(\vec{z}) $. We have that $ \vec{z}_1 $ still satisfy the constraints of (\ref{eq:MinOptGenCDS}).

                We can repeat this process until we obtain $ \vec{z}_n $ such that every $ V \in \mathcal{V}_n $ has $ d_{V} \leq 1 $. We have that $ \mathcal{V}_n $ is a combinatorial multivector field.
            \end{proof}
    
            From the previous result, we can always find parameters $ \alpha $, and $ \beta $ to obtain a combinatorial multivector field. But it might gives mediocre results where all multivectors are really small.

            We will discuss about the cost of a convex multivector. As said earlier, the minimization problem will prefer to set a low number of variables to $ 1 $. By the convexity constraints, we need to set more variables to $ 1 $. Consider $ \sigma < \tau \in $ with $ \dim \tau - \dim \sigma > 1 $ and we suppose that all costs are positives.  When $ z(\sigma, \tau ) = 1 $, it implies that $ z(\rho_i, \rho_j) = 1 $ with $ \sigma \leq \rho_i < \rho_j \leq \tau $  by the constraints of (\ref{eq:MinOptGenCDS}) . For a convex set, the real cost is
            \begin{equation}\label{eq:convexCost}
                \sum_{\sigma \leq \rho_i < \rho_j \leq \tau} c(\rho_i, \rho_j).
            \end{equation}
            Therefore, the minimization will prefer to have smaller convex set. One can choose a $ \beta > 0 $ to reduce this effect. Another way to solve this problem is by changing the objective function. We update the cost by 
            \begin{equation}
                c''(\sigma, \tau) = c'(\sigma, \tau) - \sum_{\sigma < \rho < \tau} c(\sigma, \rho).
            \end{equation}
            This approach will reduce the cost of the convex from (\ref{eq:convexCost}), and we can obtain bigger multivector.

            This model is a generalization of the model from \cite{arDDC_FiniteVecField}. It has the advantage to obtain any combinatorial multivector field. We also have two parameters $ \alpha $ and $ \beta $ where $ \alpha $ affects the quantity of multivector with only one simplex and $ \beta $ affects the size of multivector.
            
            Finally, one more hurdle is to solve an optimization problem with binary variables is Np-hard. We tried to solve the problem in the subsection \ref{ss:OverviewMain}. After seven days of solving (\ref{eq:MinOptGenCDS}), we still do not have an optimal solution for this problem.
        
        \subsection{Model 2 : One Toplex per Multivector}

            We want to simplify the variables and the constraints to obtain a model where all solutions induce a combinatorial multivector field. Moreover, we want to this model to be non-parametric. We remove the $ \alpha $ and the $ \beta $ parameters from the optimization problem (\ref{eq:MinOptGenCDS}).

            We change the variables to $ z(\sigma, \tau) $ such that $ \sigma < \tau \in T $. We minimize the same objective function $ f(\vec{z}) = \vec{c} \cdot \vec{z} $ where $ \vec{c} $ is defined by (\ref{eq:CostVect}).

            For the constraints, we change $ Dz \geq \vec{1} $ to $ \sum_{\sigma < \tau_i} z(\sigma, \tau_i) = 1 $ for each $ \sigma \in K \setminus T $. With this new constraint, we want to have that each simplex is associated to a unique toplex. To ensure the convexity of multivector, we simplify the set of constraints of the previous model. We only take the constraint $ z(\sigma_i, \tau) - z(\sigma_j, \tau) $ for each triplet $ \sigma_i < \sigma_j < \tau \in T $ and $ \dim \sigma_i + 1 = \dim \sigma_j $. If $ z(\sigma_i, \tau) = 1 $, then only way to satisfy this constraint is $ z(\sigma_j, \tau) = 1 $. This implies that $ \sigma_i $ and $ \sigma_j $ will be in the same multivector. We obtain the following optimization problem :
            \begin{equation}    \label{eq:MinOptMv}
                \begin{aligned}
            		& \underset{z(\sigma, \tau) \in \{ 0, 1 \} }{\text{minimize}}
            		& & f(\vec{z}) = \vec{c} \cdot \vec{z} \\
            		& \text{subject to}
                    & & \sum_{\sigma < \tau_i} z(\sigma, \tau_i) = 1 \qquad \text{for } \sigma \in K \setminus T \\
                    & & & z(\sigma_i, \tau) - z(\sigma_j, \tau) \leq 0 \, \text{for } (\sigma_i, \sigma_j, \tau).
        		\end{aligned}
            \end{equation}   
            where the triplet $ (\sigma_i, \sigma_j, \tau) $ satisfy $ \dim \sigma_i + 1 = \dim \sigma_j $, and $ \sigma_i < \sigma_j < \tau $.
    
            We need the next Lemma to show that the set of solutions of (\ref{eq:MinOptMv}) is not empty.
            \begin{lem}\label{lem:shavMultivec}
                Let $K$ be a simplicial complex. There exists a combinatorial multivector field $ \mathcal{V} $ such that for each $ V \in \mathcal{V} $, $ V $ has a single toplex $ \tau $, and for each $ \sigma \in V $, we have $ \sigma < \tau $.
            \end{lem}
            \begin{proof}
                We take any combinatorial multivector field $ \mathcal{V} $ such that each multivector $ V $ has at least one toplex. This is always possible, because we can take a single multivector covering each connected components of $K$. We take any $ V \in \mathcal{V} $ such that it has strictly more than one toplex. We choose a single toplex $ \tau \in V $. We divide $ V $ in two subsets $ V_1 $ and $ V_2 $. We define $ V_2 := \{ \sigma \in V \mid \sigma \text{ has a single maximum and its } \tau \} $, and $ V_1 := V \setminus V_2$. We need to show that $ V_1 $ and $ V_2 $ are convex. Let $ x, z \in V_2 $ such that $ x \leq z $. Let $ y \in V $ such that $ x \leq y \leq z $. We have that $ y $ has the same maximum $ x $ and $ z $. This implies that $ y \in V_2 $, and $ V_2 $ is convex. Let $ x, z \in V_1 $ such that $ x \leq z $. Let $ y \in V $ such that $ x \leq y \leq z $. We have that $ z $ has either more than one maximum or it has a single maximum, if it is or not $ \tau $. Moreover, $ y $ has either the same maximum has $ z $ or more. This implies that $ y \in V_1 $, and $ V_1 $ is convex.
    
                We define a new combinatorial multivector field $ \mathcal{V}' $ with $ \mathbb{V} $ by removing $ V $, adding $ V_1 $ and $ V_2 $. We can repeat this process until there is only one toplex in each multivector and we obtain the expected result.
            \end{proof}
    
            Now, we can show the main result of this section.
            \begin{thm}
                The solutions of (\ref{eq:MinOptMv}) induce a combinatorial multivector field $ \mathcal{V} $.
            \end{thm}
            \begin{proof}
                First, we need to show that the set of solutions of (\ref{eq:MinOptMv}) is non-empty. We apply the previous Lemma (\ref{lem:shavMultivec}) to obtain a combinatorial multivector field $ \mathcal{V} $. We assign $ z(\sigma, \tau) = 1 $, if $ \sigma $ and $ \tau $ are in the same multivector. Each multivector has a single toplex and each simplex is in a single multivector. Therefore, for all $ \sigma \in K \setminus T $, there exists a unique $ \tau$ such that $ z(\sigma, \tau) = 1 $ . This implies that the first set of constraint is satisfy. For the second set of constraints, we have $ z(\sigma_1, \tau) - z(\sigma_2, \tau) \leq 0 $ for a triplet $(\sigma_1, \sigma_2, \tau)$ with $ \sigma_1 < \sigma_2 < \tau $, and $ \dim \sigma_1 + 1 = \dim \sigma_2 $. The only way to not satisfy this condition is when $ z(\sigma_1, \tau) = 1 $, and $z(\sigma_2, \tau) = 0$. This implies that $ \sigma_1 $ and $ \sigma_2 $ will be in different multivector. This is impossible because the multivectors are convex. Therefore, for any simplicial complex, the set of solutions of (\ref{eq:MinOptMv}) is non-empty.
        
                Let $ \vec{z} $ be a solution from (\ref{eq:MinOptMv}). For each toplex $ \tau \in T $, we define a combinatorial multivector field with
                \begin{equation}    \label{eq:Model2MultVec}
                    V_\tau = \{ \sigma \in K \mid \sigma \leq \tau \text{ and } z(\sigma, \tau) = 1 \} \cup \tau.
                \end{equation}
                For each simplex $ \sigma \in K $, there exists a unique $ \tau $ such that $ z(\sigma, \tau) = 1 $ by the first constraint. We obtain $\sigma \in V_\tau $. For each toplex $ \tau $, we have $ \tau \in V_{\tau} $ by definition. Let $ V_{\tau_1} $ and $ V_{\tau_2} $ be two multivectors in $ \mathcal{V} $. Suppose there exists $ \sigma \in V_{\tau_1} \cap V_{\tau_2} $. If $ \sigma \not\in T $, then we have that $ z(\sigma, \tau_1) = z(\sigma, \tau_2) = 1$. This is impossible by the first set of constraints. By construction, we have that each toplex is in a different multivector. Therefore, $ V_{\tau_1} \cap V_{\tau_2} = \emptyset $ for all $ \tau_1, \tau_2 \in T$. We obtain that the set of $ V_{\tau} $, for $ \tau \in T $, define a partition.
        
                We need to show that each $ V_{\tau} $ is convex. Let $ \sigma_1, \sigma_3 \in V_{\tau} $ such that $ \sigma_1 < \sigma_3 $. This implies that $ z(\sigma_1, \tau) = 1 $ Let $ \sigma_2 \in K $ such that $ \sigma_1 \leq \sigma_2 \leq \sigma_3  $. We show by induction on dimension the simplex $ \sigma_2 $. Let $\dim \sigma_1 = d $, and $ \dim \sigma_2 = d + 1 $. We have the following constraint $ z(\sigma_1, \tau) - z(\sigma_2, \tau) \leq 0$ , because $ (\sigma_1, \sigma_2, \tau) $ is a triplet that satisfy the condition. We have that $ z(\sigma_1, \tau) = 1 $. Then, the only way to satisfy the constraint is $ z(\sigma_2, \tau) = 1 $. Then, $ \sigma_2 \in V_{\tau} $. If $ \dim \sigma_2 = d + i $ with $ i < 1 $, there exists a sequence of simplices $ \beta_j $ such that $\sigma_1 = \beta_1 < \beta_2 < \ldots < \beta_{n-1} < \beta_n = \sigma_2 $ with $ \dim \beta_j + 1 = \dim  $. By induction, we have that $ z(\beta_j, \tau) = 1 $ for all $j$, and we obtain $ \sigma_2 \in V_{\tau} $. Therefore, $ V_{\tau} $ is convex. Finally, the solution of (\ref{eq:MinOptMv}) induce a combinatorial multivector field.
            \end{proof}
        
            Now, we want to rewrite the problem (\ref{eq:MinOptMv}) into matrix notation. We define a matrix for each type of constraints. We assign an index $ j $ for each variable $ z(\sigma, \tau) $, and let $ m $ be the number of variables $z$. Let $ O_{n_1 \times m} $ be a matrix with binaries values such that $ n_1 = \vert K \setminus T \vert $. We assign an index $ i_1 $ for each $ \sigma \in K \setminus T $. Let $ \sigma_{i_1} $ be a simplex associated to the $i_1 $th row. The entries of $ O $ are :
            \begin{equation*}
                O_{i_1, j} := \begin{cases}
                    1 & \text{if } \sigma_{i_1} = \sigma \\
                        0 & \text{otherwise}
                \end{cases}.
            \end{equation*}
            
            Let $ M_{n_2 \times m} $ be a matrix with values in $ \{-1, 0 ,1 \} $ such that $ n_2 $ is the number of triplets $ (\sigma_1, \sigma_2, \tau ) $ defined in (\ref{eq:MinOptMv}). Let be a triplet $(\sigma_1, \sigma_2, \tau)$ associated to the $ i_2 $th row. The entries of $M$ are:
        \begin{equation*}
            M_{i_2, j} := \begin{cases}
                1 & z_j \text{ is associated to } (\sigma_1, \tau) \\
                -1 & z_j \text{ is associated to } (\sigma_2, \tau) \\
                0 & \text{otherwise}
            \end{cases}.
        \end{equation*}
    
        Finally, we can rewrite the linear problem (\ref{eq:MinOptMv}) into the matrix form :
        \begin{equation}\label{eq:MinOptMvMatrix}
            \begin{aligned}
    		      & \underset{\vec{z} \in \{0, 1\}^m }{\text{minimize}}
    		      & & f(z) = \vec{c} \cdot \vec{z} \\
    		      & \text{subject to}
                & & O\vec{z}  = \vec{\mathbf{1}}_{n_1} \\
                & & & M\vec{z}  \leq \vec{\mathbf{0}}_{n_2} 
        	\end{aligned}
        \end{equation}
        We note that each solution of (\ref{eq:MinOptMvMatrix}) induce a combinatorial multivector field $ \mathcal{V} $ where $ \vert T \vert = \vert \mathcal{V} \vert $.

%%%%%%%%%%%%%%%%%%%%%%%%%%%%%%%%%%%%%%%%%%%%%%%%%%
    
        We study the time complexity of solving the (\ref{eq:MinOptMvMatrix}).
        \begin{lem}\label{lemNbrVarEqMinTop}
            Let $m$ be the number of variables and $ n $ the number of equations of (\ref{eq:MinOptMvMatrix}). We suppose that each toplex of $ K $ have the same dimension $d > 0$.
            \begin{gather*}
                m = (2^{d}-2) \vert T \vert, \\
                n = (\vert K \vert - \vert T \vert) + (d+1)(2^d - 1) \vert T \vert.
            \end{gather*}
        \end{lem}
        \begin{proof}
            The number of simplices in a $d$-simplex is $2^d - 1$. For each simplex $ \sigma \in K \setminus T $, we match it to a $ \tau \in T $ such that $ \sigma < \tau $. We have $ (2^{d} - 2) \vert T \vert $ variables in (\ref{eq:MinOptMvMatrix}).
            
            We have two sets of constraint. For the matrix $ O $, we have $ \vert K \vert - \vert T \vert $. For the second matrix $ m $, we need to consider each triplet $ \sigma_i < \sigma_j < \tau $ with $ \dim \sigma_i + 1 = \dim \sigma_j $ and $ \tau \in T $. The number of equations for a fixed toplex is equal to the number of admissible matching in the sense of Forman. From \cite{arDDC_FiniteVecField}, we have $ (d+1)(2^d-1)$ for each toplex. Finally, $ n = (\vert K \vert - \vert T \vert) + (d+1)(2^d - 1) \vert T \vert $. 
        \end{proof}
    
        As said earlier, to solve the optimization of a linear problem with binary variables is Np-hard. Under some circumstances, the problem (\ref{eq:MinOptMvMatrix}) can be solved in polynomial time. In general, the method to solve an linear optimization problem with integer variables have two steps. First, we relax the variables to real values, and we solve the relaxed problem with classical method such as simplex method. The solution might have non-integer values. The second step is to transform the non-integer solution to an integer solution. The first step is consider to be solve in polynomial time in average\cite{boOptNum}. But, the second step is non-polynomial. We can sometimes skip the second step under some hypothesis. Per example, if the matrix of constraint is totally unimodular and the constant on the right have interger values, then the solution have integer values\cite{arHoffUniModMat}. In our problem, the constraint matrix of (\ref{eq:MinOptMvMatrix}) is not totally unimodular. But, with our experiments, we obtain each time an integer solution after the first step. Therefore, we state this conjecture.  
        \begin{conj}\label{conj:AlgoCmp}    
            Consider the objective function $ f(z) $ where all $ c_j $ are different. Then, the optimal solution of the relaxed problem of (\ref{eq:MinOptMvMatrix}) with positives real variables is the same has the optimal solution of (\ref{eq:MinOptMvMatrix}).
        \end{conj}
        
        We need to have some hypothesis on the cost of $ f(z) $. Because, if all the costs are the same, we can find an optimal solution with non-integer values.  

\section{Complete Pipeline and interpretation}
The optimization problem in section 3 assumes that we have a simplicial complex with a vector at each simplex of our simplicial complex. The goal is to create a multivector field and to try to understand what could be the dynamics in the region where we want to study the cloud of vector.
Here we present the key steps and their importance in the study.

 \subsection{Construction of a simplicial complex on data}
 The dataset has two components, an initial point $ x \in \mathbb{R}^n $ and a direction $ \dot{x} \in \mathbb{R}^n$ associated to $x$.
We need first to decide where we want to study the system. Raw data in general are not well aligned and are not really easy to study, we need to clean them first and then decide a relevant region where we can study these data. We call that region the region of interest(ROI). When we have our ROI, we now have to choose how to construct the simplicial complex with those cloud of vectors such that the minimization problem gives us at least a solution with a nice interpretation and also we try to reduce the number of size of the simplicial complex. Methods for the construction include the use of Dowker complexes and/or k-mean clustering (see \cite[Setion 3]{arDDC_FiniteVecField}).

\subsection{Multivector field construction}

When we have our simplicial complex with the value of vectors $ V(\sigma) $ for each simplices, we now need to construct a multivector field. We have two models and each of them presents some advantages and also some limitations.
\subsubsection{One toplex per multivector's model}
This model provides a simple an efficient approach for the construction of combinatorial multivector fields. Here are some advantages:
\begin{itemize}
    \item \textbf{Simplify variables and constraints:} By assigning one toplex per multivector, it reduces the number of variables and constraints, making the computations faster.
    \item \textbf{Guarantees convexity:} The constraints are designed in a way that at the end of the computation a multivector field is obtained.
\end{itemize}
However this approach has also some limitations:
\begin{itemize}
    \item \textbf{Hard to generalize:} This model restriction limits the range of possibility for the construction of a combinatorial multivector field, making it less flexible.
    \item \textbf{Critical multivectors:} Some complex invariant sets cannot be represented using this model, so it creates sometimes inside those invariant sets artifacts that are considered as critical multivectors even when they should not exist in the invariant set.
\end{itemize}
In general this model is for applications where the priority is to quickly understand the general behavior of the system and may not work well for the case of complex dynamical systems.

\subsubsection{Generalization of the CDS Model}

This model extend the previous one allowing more flexibility and control over the construction of multivector fields. Its main advantages are:
\begin{itemize}
    \item \textbf{Broader applicability:} This model can construct a more flexible type of combinatorial multivector field, making it fit more complex dynamics.
    \item \textbf{Richer dynamical representations:} Since the model is more flexible, it is possible to capture much more complex invariant sets and study more complex systems.
\end{itemize}
Unfortunately, this model at the moment brings a lot of challenges:
\begin{itemize}
    \item \textbf{Convexity issues:} The resulting solutions of the model are not always multivector fields, because the convexity requirement for multivectors may not be satisfied. it requires some post processing. Propagation method (see\cite{Woukeng_2024}) may be used in this case as a post processing method.
    \item \textbf{Complexity of the model:} The need for parameter tuning and additional constraints increase the computational cost in general making it slower than the previous model.
    \item \textbf{NP-hard optimization:} The optimization problem involves binary variables, which significantly increases the difficulty of the computation.
\end{itemize}
    This model is then a better fit when you want to study a more complex systems but will requires a higher computational cost and some post processing methods.

\subsection{Morse decomposition and interpretations}
By constructing multivector fields, we are interested in the dynamics generated by the study of the induced combinatorial dynamical system. Having a Morse set with the right Conley index gives us some information about the possible nature of the combinatorial solutions, and from them we can deduce some possible behaviors that the data can exhibit.\\
Here are some invariants sets, their Conley indices and what they could probably be:
\subsubsection{Periodic orbits}

For the case of periodic orbits in general from \cite{mrozek2022combinatorial} we have this formula, and that should work in any dimension.
If $\cA$ represents a connected isolated invariant set, that is a Morse set from the Morse decomposition such that for $r = 0$ or $r = 1$, we have
\begin{align*}
    \dim H_{2n+r}(\cl \cA, \mo \cA) = \dim H_{2n + 1 + r} (\cl \cA, \mo \cA)~~~~~
\text{for all }
n \in \ZZ ,
\end{align*}
where not all of these homology groups are trivial. Then, the set $\cA$  may be the Morse set of a non-trivial periodic orbit.\\
   
\subsubsection{Attracting fixed point or Attractors}        

For the case of an attracting fixed point and some attractors in general from \cite{forman1998combinatorial} we have this formula, and that should work in any dimension.
If $\cA$ represents a connected isolated invariant set, that is a Morse set from the Morse decomposition, we have
\begin{align*}
    \dim H_{n}(\cl \cA, \mo \cA) &=  1
\text{ if n=0, }\\
0 & \text{ otherwise}
\end{align*}  

\subsubsection{Repelling fixed point or Repellers}        

For the case of a repelling fixed point and some attractors in general from \cite{forman1998combinatorial} we have this formula, and that should work in any dimension. If $d$ is the dimension of the space and $\cA$ represents a connected isolated invariant set, that is a Morse set from the Morse decomposition, we have
\begin{align*}
    \dim H_{n}(\cl \cA, \mo \cA) &=  1
\text{ if n=d, }\\
0 & \text{ otherwise}
\end{align*}

\section{Experimentation}
    \subsection{Vanderpol}
        We consider the following system :
        \begin{equation}    \label{eq:Vanderpol}
            \begin{cases}
                \frac{dx}{dt} =  y \\
                \frac{dy}{dt} = y(1-x^2) - x 
            \end{cases}
        \end{equation}
        This dynamical system have a repulsive fixed point at $ (0, 0) $ and an attractive orbit around $ (0, 0) $. We start with $ 1000 $ data points randomly taken in $ [-7, 7] \times [-7, 7] $ and each point has a vector computed by (\ref{eq:Vanderpol}). We apply the k-means method to obtain $ 50 $ clusters, and we build the Delaunay complex $K$ on the set of clusters. We compute to value of $ V(\sigma) $ in the same way as subsection \ref{ss:OverviewMain}.

        We solve the minimization problem of (\ref{eq:MinOptMvMatrix}), and the optimal solution induce the combinatorial multivector field $\mathcal{V}$ at Figure (\ref{fig:cmvfVanderpol}). We have a single critical multivector denoted by $86$. This multivector is repulsif. We have a single cycle $ S $ with ten multivectors in the third Figure of (\ref{fig:cmvfVanderpol}). The exit set of $ S $ is empty. We retrieve that $ S $ has a similar dynamic has an attractive orbit.
        
        %%%%%%%%%%%%%%%%%%%%%%%%%%%%%%%%%%%%%%%%%%%
         \begin{figure}
  	         \center
             \includegraphics[height=4cm, width=6cm, scale=1.00, angle=0 ]{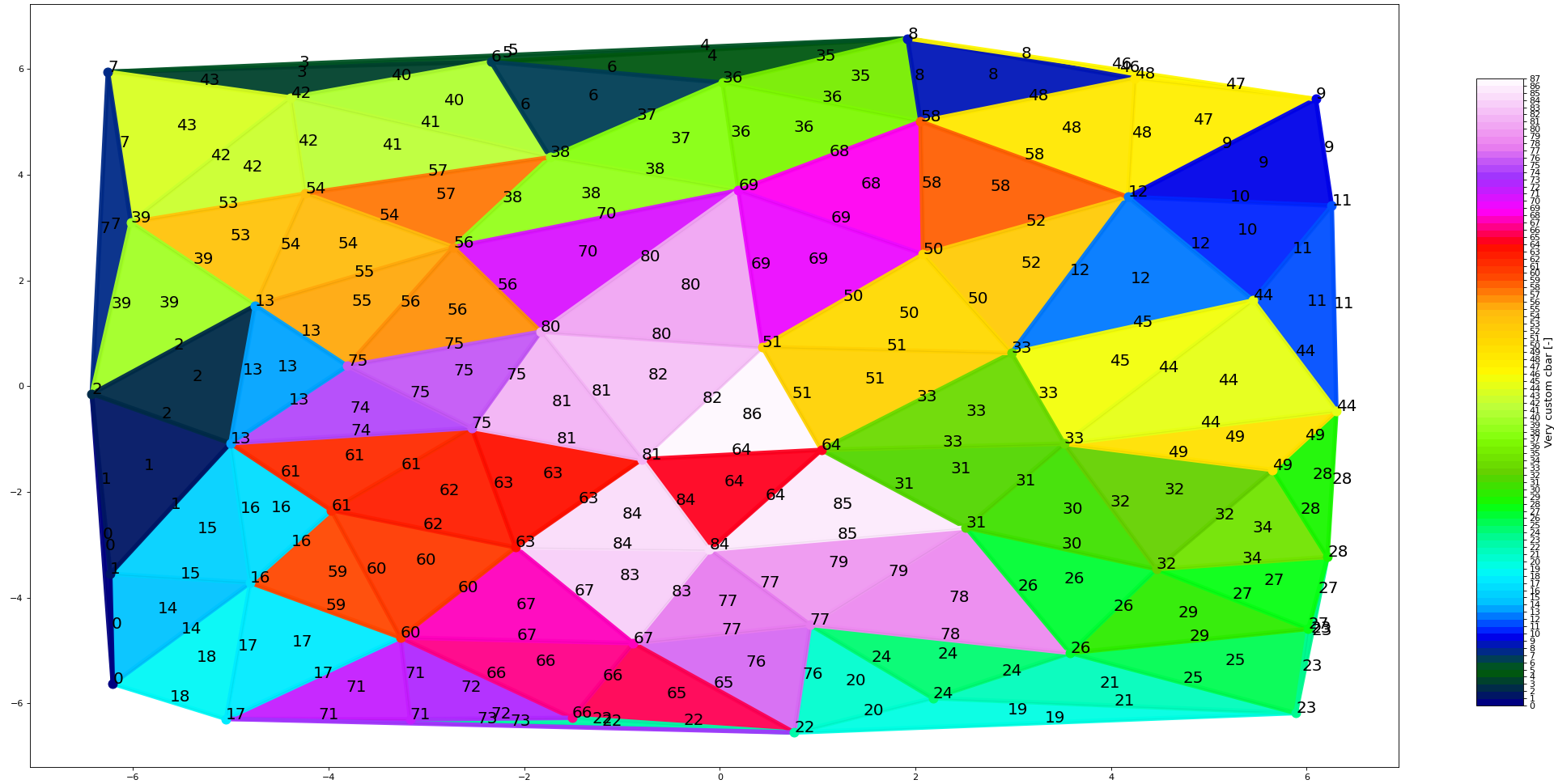}
             \includegraphics[height=4cm, width=6cm, scale=1.00, angle=0 ]{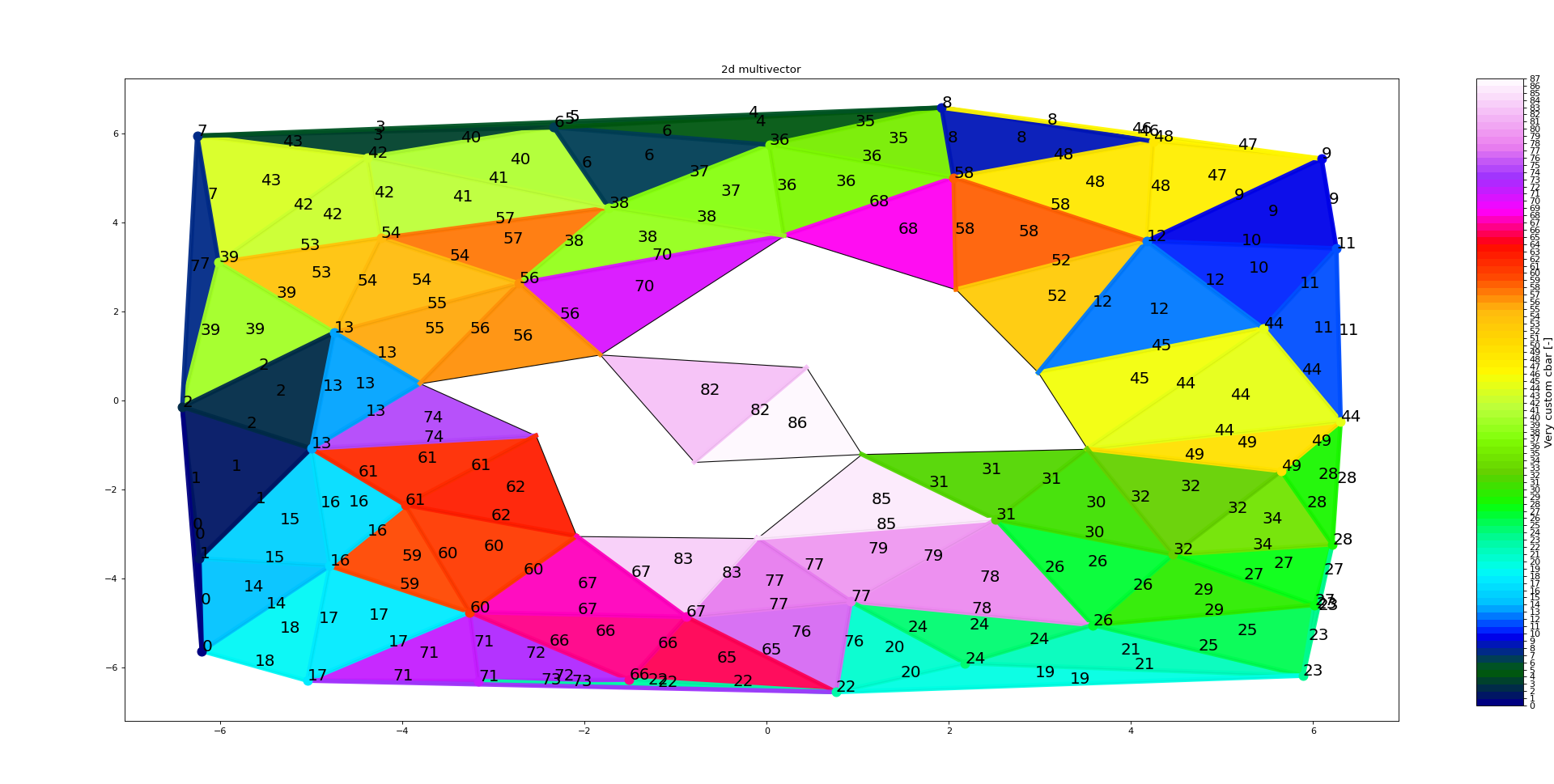}
             \includegraphics[height=4cm, width=6cm, scale=1.00, angle=0 ]{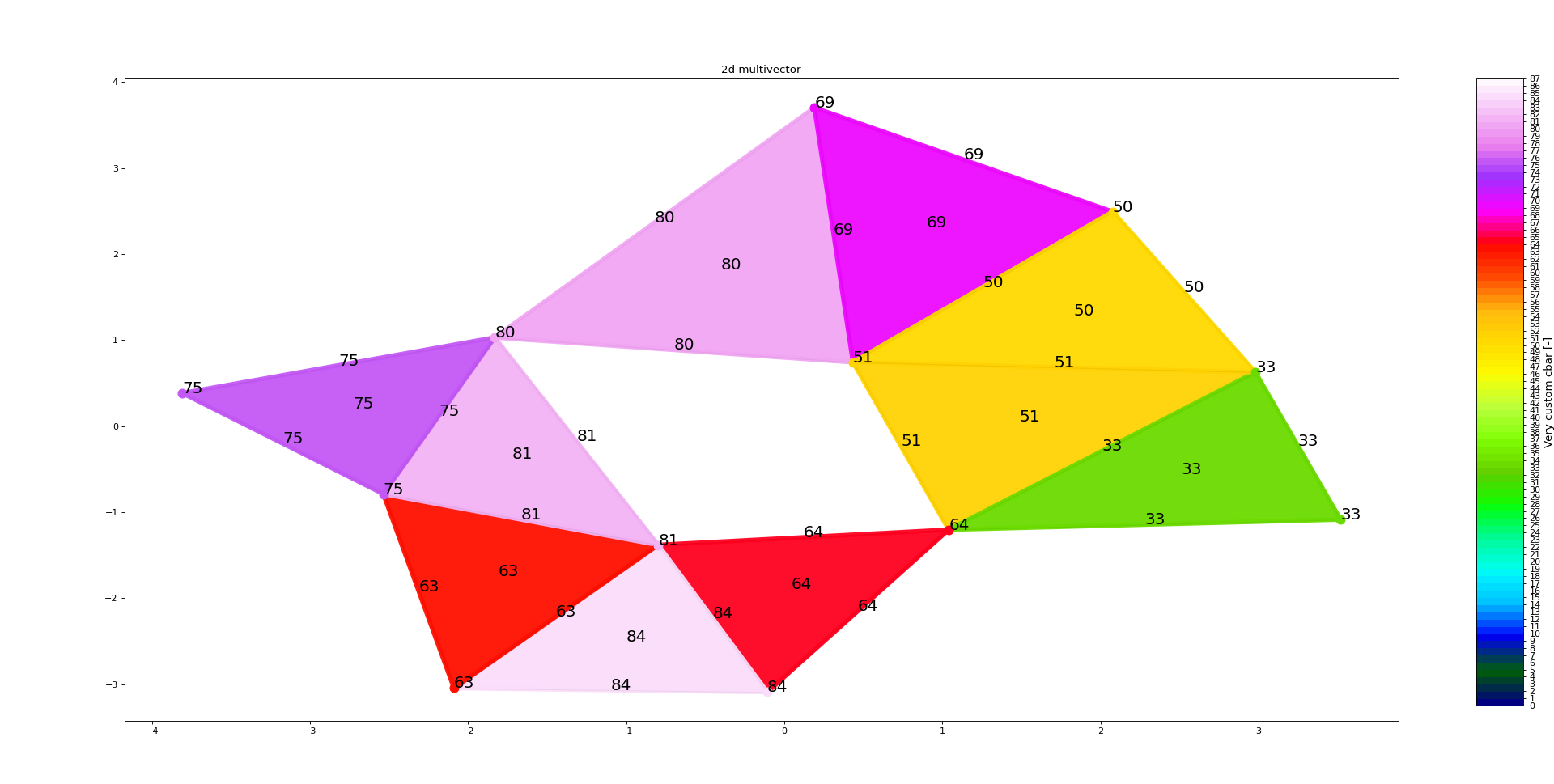}           
             \caption{The first Figure is the combinatorial multivector field  $\mathcal{V}$ obtained from the Vanderpol oscillator's equation. The second Figure is the gradient part of $\mathcal{V}$. The third Figure is the only strongly connected components of $ \mathcal{V} $.}	
             \label{fig:cmvfVanderpol}
     \end{figure}
        %%%%%%%%%%%%%%%%%%%%%%%%%%%%%%%%%%%%%%%%%%%%
        
    \subsection{Lorenz Attractors}
        We consider the following system :
        \begin{equation}
            \begin{cases}
                \frac{dx}{dt} = 10(y-x) \\
                \frac{dy}{dt} = 28x-xz-y \\
                \frac{dz}{dt} = xy - \frac{8}{3}z            
            \end{cases}
        \end{equation}
        This dynamical system exhibits a strange attractor. We want to be able to retrieve this kind of information. 
        For the datasets, we take a linear approximation of a trajectory with the initial point $ x_0 = (0.00, 1.00, 1.05) $ and $ \Delta t = 0.2 $ :
        \begin{equation*}
            x_{i+1} = x_i + \Delta t \dot{x_i} \quad i = 0,1,2 \ldots 999.
        \end{equation*}
        To reduce the number of points we apply k-means clustering to reduce to $ 75 $ points and we apply the same procedure as the previous subsection. The optimal solution of (\ref{eq:MinOptMvMatrix}) induce a combinatorial multivector field with $ 380 $ multivectors. The number of critical multivectors is $ 21 $ where $ 8 $ of them has a single toplex. We obtain a single strongly connected component $S$ with $ 351 $ multivectors and contains ten critical multivectors. We have that $ \frac{351}{380} \sim 92.37 \% $ of multivectors are in $S$. We have that the exit set of $ S $ is empty, and this implies that $ S $ has dynamics similar to an attractor. 

        %%%%%%%%%%%%%%%%%%%%%%%%%%%%%%%%%%%%%%%%
       \begin{figure}
  	         \center
             \includegraphics[height=5cm, width=6cm, scale=1.00, angle=0 ]{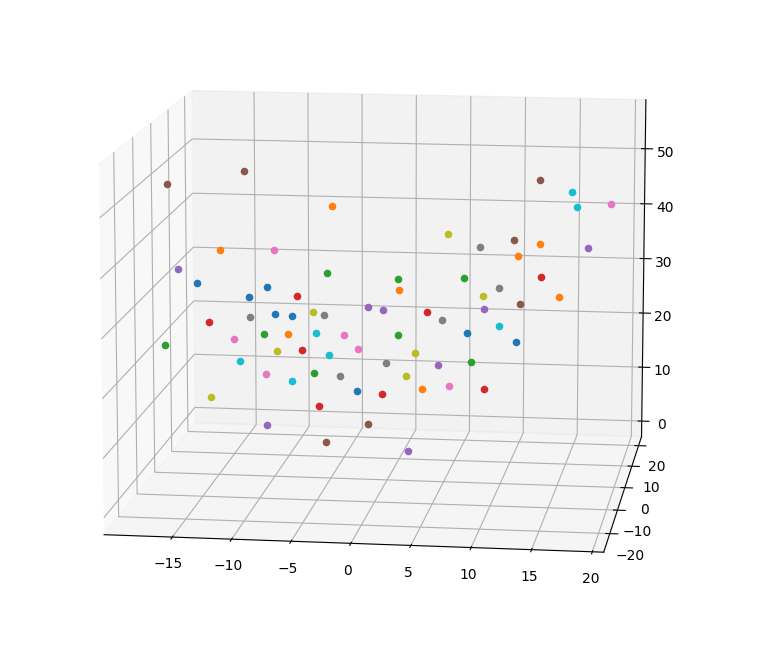}
             \includegraphics[height=5cm, width=6cm, scale=1.00, angle=0 ]{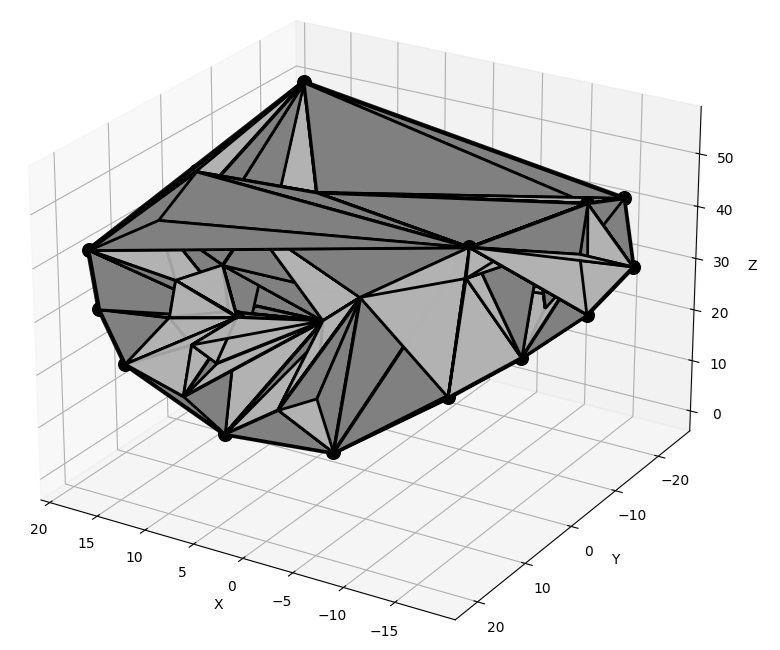}          
             \caption{On the left, we have the set of cluster and on the right, we have the strongly connected component.}	
             \label{fig:expLorenzAtt}
     \end{figure}
        %%%%%%%%%%%%%%%%%%%%%%%%%%%%%%%%%%%%%%%%
    
\section*{Competing interests}
   Research of D.W. is partially supported by the Polish National Science Center under Opus Grant No. 2019/35/B/ST1/00874.\\
   Research of D.D.C is supported by Mathematics Department of Université de Sherbrooke.\\
   The authors have no conflict of interest to declare that are relevant to the content of this article.

\newpage
\bibliographystyle{plain}
\bibliography{references,references_DDC}

\end{document}